\documentclass[11pt,a4paper]{amsart}
\usepackage[T1]{fontenc}
\usepackage{lmodern}
\usepackage[margin=29mm,headheight=14pt,headsep=8mm]{geometry}
\usepackage{mathtools,amssymb}
\usepackage{microtype}
\usepackage{enumitem}
\usepackage[hidelinks,unicode]{hyperref}

\hypersetup{
  pdftitle={One-point metrizable coarsenings: gauges and local metric preservation},
  pdfauthor={First Author Name; Second Author Name},
  pdfsubject={General topology and metric spaces},
  pdfkeywords={coarser topology, metric, one-point extension, continuous gauge,
    local isometry, complete metrizability}
}

\newtheorem{theorem}{Theorem}[section]
\newtheorem{proposition}[theorem]{Proposition}
\newtheorem{lemma}[theorem]{Lemma}
\newtheorem{corollary}[theorem]{Corollary}
\theoremstyle{definition}
\newtheorem*{definition}{Definition}
\newtheorem{example}[theorem]{Example}
\theoremstyle{remark}
\newtheorem{remark}[theorem]{Remark}

\numberwithin{equation}{section}
\newcommand{\N}{\mathbb N}
\newcommand{\R}{\mathbb R}
\newcommand{\Ga}{\mathcal G_a(\tau)}
\newcommand{\La}{\mathcal L^m_a(\tau)}
\newcommand{\Bor}{\mathcal B}
\newcommand{\Cb}{C_b}
\newcommand{\ev}{\operatorname{ev}}
\newcommand{\id}{\operatorname{id}}
\newcommand{\clos}[1]{\overline{#1}}
\setlist[enumerate,1]{label=\textup{(\roman*)},leftmargin=2.3em,itemsep=3pt,topsep=5pt}
\title[One-point metrizable coarsenings]{One-point metrizable coarsenings:\\
gauges and local metric preservation}

\author{Ahmad ja'afari kalvan$^{*}$}
\address{Department of Mathematics, Tarbiat Modares University, 14115-134, Tehran, Iran}
\email{ahmad.jaafari@modares.ac.ir}
\thanks{$^{*}$Corresponding author}

\author{Ehsan Shahoseini$^{1}$}
\address{School of Mathematics, Institute for Research in Fundamental Sciences (IPM), P.O. Box: 19395-5746, Tehran, Iran}
\email{shahoseini@ipm.ir}
\thanks{$^{1}$ The second author's research was supported by a grant from IPM.}

\date{}

\begin{document}

\begin{abstract}
Let $(X,\tau)$ be metrizable and let $a\in X$. We give a constructive
account of metrizable topologies $\sigma\subseteq\tau$ that agree with
$\tau$ on $X\setminus\{a\}$. Applying Hausdorff's classical metric
collapse construction, for every noncompact $(X,\tau)$ and every
compatible metric $d$ we obtain a strict coarsening with a metric
$p\le d$ that agrees with $d$ on a common neighborhood of each point
other than $a$. A prescribed countably infinite closed discrete set
$\{x_n:n\in\N\}\subseteq X\setminus\{a\}$ can be made to satisfy
$p(a,x_n)\le\lambda_n$ for any positive null sequence $(\lambda_n)$.
The resulting metric is greatest among the metrics dominated by $d$
that satisfy these bounds, and is complete whenever $d$ is complete.
We exhibit its realization as a classical metric quotient. We also
represent all localized metrizable coarsenings by continuous scalar
gauges using a standard cone metric. Inclusion is expressed by the
cofinal comparison of sublevel sets familiar from extension-trace
theory, while pointwise maximum and minimum realize finite joins and
meets. A closed-discrete criterion detects strictness. Standard
preservation results for Borel structure, complete metrizability, and
Polishness, together with function-space and local-field examples,
complete the account.
\end{abstract}

\maketitle

\section{Introduction}

Weakening a topology while retaining metrizability is a basic problem in
general topology. Compact Hausdorff topologies admit no strictly weaker
Hausdorff topology; this is part of the classical theory of minimal
Hausdorff spaces \cite{SW}. In the opposite direction, Gruenhage,
Tkachuk, and Wilson \cite{GTW} proved that every noncompact metrizable
space admits a weaker metrizable topology that is nowhere locally compact.
Here we consider weakenings whose effect is confined to one prescribed point,
and ask how much of a given metric can be retained elsewhere.

Throughout the paper, $X$ is a nonempty set, $\tau$ is a metrizable
topology on $X$, and $a\in X$. We write $B_d(x,r)$ for an open metric
ball and put $\N=\{1,2,\ldots\}$. A set is called \emph{closed discrete}
when it is closed in the ambient space and has the discrete subspace
topology.
Enumerations of countably infinite closed discrete sets are taken
without repetitions.

\begin{definition}
A \emph{one-point localized metrizable coarsening} of $\tau$ at $a$ is a
metrizable topology $\sigma$ on $X$ such that
\begin{equation}\label{eq:localization}
  \sigma\subseteq\tau,
  \qquad
  \sigma|_{X\setminus\{a\}}=\tau|_{X\setminus\{a\}}.
\end{equation}
The family of all such topologies, ordered by inclusion, is denoted by
$\La$. The coarsening is \emph{strict} if $\sigma\subsetneq\tau$.
\end{definition}

Since $\{a\}$ is closed in every topology under consideration, its
complement is open. Thus equality in \eqref{eq:localization} implies
equality of the ambient local topologies at each $x\ne a$.

There are direct precedents for this setting. Kuba \cite{Kuba} studied
the complete lattice of Hausdorff topologies on $\R$ coarser than the
Euclidean topology and agreeing with it on $\R\setminus\{0\}$, including
its metrizable members and chains. A recent expanded version
\cite{Kuba2026} incorporates this development and adds a 2026 supplement
on counting with respect to weight and related cardinal questions. The scope here is
different: instead of restricting the underlying space to $\R$, we work
with an arbitrary metrizable space $(X,\tau)$ and focus throughout on
metrizable one-point coarsenings. In particular, we give explicit
compatible metrics that preserve the original metric locally away from
the distinguished point, an extremal metric under prescribed distance
bounds, and a gauge representation that describes the inclusion order of
all localized metrizable coarsenings.

Moreover, if $Y=X\setminus\{a\}$ and $a$ is nonisolated in $\sigma$,
then $(X,\sigma)$ is a one-point metrizable extension of $Y$. Extension
traces and their cofinal order are described by Henriksen, Jano\v{s},
and Woods \cite[Theorems~3.3, 3.5, and~3.9]{HJW}; the trace metrization
theorem there is attributed to Alexander. For related work on one-point
metrizable extensions in the locally compact metrizable setting, see
Koushesh \cite[Section~1]{Koushesh}. The gauge description below gives an
explicit scalar realization of this framework subject to the
additional condition $\sigma\subseteq\tau$.

The metric construction has a classical source as well. Hausdorff
\cite[p.~354, formula~(0)]{Hausdorff} used the minimum of the original
distance and the sum of the distances to a closed set. Ishiki
\cite[Section~2.3, Propositions~2.6 and~2.9]{Ishiki} gives its metric
quotient formulation and proves preservation of completeness. We apply
this construction to a closed set in a product space so that the
quotient has underlying set $X$ and the change of topology is localized
at $a$.

The main contributions of the paper are fourfold. First, for every
noncompact metrizable space and every prescribed point, we obtain a
strict localized metrizable coarsening by an explicit metric dominated
by a given compatible metric; away from the distinguished point the two
metrics agree on a common neighborhood of each point. Second, the same
construction permits prescribed upper bounds on the distances from the
distinguished point to a closed discrete sequence, and the resulting
metric is maximal in pointwise order among all metrics satisfying those
bounds. Third, continuous scalar gauges give explicit compatible
metrics for all localized metrizable coarsenings and encode their
inclusion order, yielding concrete finite lattice operations. Finally,
we characterize strictness by closed-discrete convergence and record
the resulting preservation of Borel structure, complete metrizability,
and Polishness. The existence construction is an application of
Hausdorff's classical collapse method, whose quotient realization is
made explicit in Remark~\ref{rem:quotient}; the subsequent gauge,
order, and strictness results place the coarsening constraint in a
single constructive framework.

\begin{theorem}[Existence with local metric preservation]\label{thm:main-existence}
Let $(X,\tau)$ be noncompact and metrizable, let $a\in X$, and let $d$
be any compatible metric. There is a countably infinite $\tau$-closed
discrete set $D=\{x_n:n\in\N\}\subseteq X\setminus\{a\}$ such that,
for every sequence $\lambda_n>0$ with $\lambda_n\to0$, there is a metric
$p$ on $X$ satisfying:
\begin{enumerate}
\item $p\le d$, and its topology $\tau_p$ belongs to $\La$ with
      $\tau_p\subsetneq\tau$;
\item for each $x\ne a$, there is a neighborhood $U_x$ open in both
      topologies such that $p(y,z)=d(y,z)$ for all $y,z\in U_x$;
\item $p(a,x_n)\le\lambda_n$ for every $n$, so $x_n\to a$ in $\tau_p$;
\item if $d$ is complete, then $p$ is complete.
\end{enumerate}
The construction works for every countably infinite $\tau$-closed
discrete set $D\subseteq X\setminus\{a\}$ fixed in advance.
No boundedness assumption on $d$ is required.
\end{theorem}

The formula is
\begin{equation}\label{eq:main-formula}
\begin{split}
  h(x)&=\inf_{n\ge1}\{\lambda_n+d(x,x_n)\},
  \qquad g(x)=\min\{d(x,a),h(x)\},\\
  p(x,y)&=\min\{d(x,y),g(x)+g(y)\}.
\end{split}
\end{equation}
The positivity of $h$ and the $1$-Lipschitz property of $g$ are the key
points. Section~\ref{sec:localmetric} proves the theorem and explains
the quotient construction. Proposition~\ref{prop:maximal-metric}
characterizes $p$ as the greatest metric $r\le d$ satisfying
$r(a,x_n)\le\lambda_n$ for every $n$.

For the representation of all localized coarsenings, arbitrary continuous
gauges are more natural. Define
\[
  \Ga=\{s:X\to[0,\infty):s\text{ is $\tau$-continuous and }
                         s^{-1}(0)=\{a\}\}.
\]

\begin{theorem}[Gauge representation]\label{thm:main-gauges}
Fix a compatible metric $d\le1$ for $\tau$. For $s\in\Ga$, define
\begin{equation}\label{eq:cone-pullback}
  \rho_s(x,y)=|s(x)-s(y)|+\min\{s(x),s(y)\}d(x,y).
\end{equation}
Then $\rho_s$ is a metric whose topology $\tau_s$ lies in $\La$, and
$\{\{s<\varepsilon\}:\varepsilon>0\}$ is a neighborhood base at $a$.
Every member of $\La$ arises in this way. Moreover,
\begin{equation}\label{eq:main-order}
  \tau_s\subseteq\tau_t
  \quad\Longleftrightarrow\quad
  \forall\varepsilon>0\ \exists\delta>0:
          \{t<\delta\}\subseteq\{s<\varepsilon\}.
\end{equation}
The topology $\tau_s$ is independent of the choice of compatible
$d\le1$. If this $d$ is complete, then $\rho_s$ is complete.
\end{theorem}

The cone formula in \eqref{eq:cone-pullback} is standard; the identical
formula on a bounded height interval is used in
\cite[Section~4]{BZ}. Our use of it records an explicit metric for every
admissible gauge. The order criterion is the cofinal comparison of local
bases familiar from extension theory. Section~\ref{sec:gauges} proves
the representation and identifies finite joins and meets. Section~\ref{sec:structure}
explains the extension-trace correspondence and gives a closed-discrete
criterion for strictness. It also records the preservation results;
complete metrizability is an application of
\cite[Theorem~2.2]{HJW}, while \eqref{eq:cone-pullback} provides an
explicit complete compatible metric. Section~\ref{sec:examples} contains
the function-space consequence and the local-field example.

\section{A coarser metric with local equality of distances}\label{sec:localmetric}

\begin{lemma}\label{lem:discrete}
If $(X,\tau)$ is noncompact and metrizable, then for every $a\in X$
there is a countably infinite $\tau$-closed discrete subset of
$X\setminus\{a\}$.
\end{lemma}

\begin{proof}
In a metrizable space, compactness is equivalent to sequential
compactness. Choose a sequence with no convergent subsequence. Its range
$E$ is infinite and has no accumulation point: otherwise first countability
would allow the choice of distinct terms, with increasing indices,
converging to an accumulation point. Thus $E$ is countably infinite,
closed, and discrete. If $a\notin E$, take $D=E$. If $a\in E$, choose
an open set $V$ with $V\cap E=\{a\}$. Then
$D=E\setminus\{a\}=E\cap(X\setminus V)$ is again closed discrete and
countably infinite.
\end{proof}

The next lemma is a gauge form of Hausdorff's construction
\cite[p.~354, formula~(0)]{Hausdorff}. For the metric quotient
formulation and its completeness, see
\cite[Propositions~2.6 and~2.9]{Ishiki}. We include the proof to record
explicitly a common ball on which the original distances are retained.

\begin{lemma}[A gauge form of Hausdorff's construction]\label{lem:minmetric}
Let $d$ be a metric on $X$ and let $g:X\to[0,\infty)$ be
$1$-Lipschitz with $g^{-1}(0)=\{a\}$. Put
\begin{equation}\label{eq:minmetric}
  p_g(x,y)=\min\{d(x,y),g(x)+g(y)\}.
\end{equation}
Then $p_g$ is a metric, $p_g\le d$, and $p_g(a,x)=g(x)$. For $x\ne a$,
the set
\begin{equation}\label{eq:equal-ball}
  U_x=B_d\bigl(x,g(x)/2\bigr)=B_{p_g}\bigl(x,g(x)/2\bigr)
\end{equation}
is open in both metric topologies, and
\begin{equation}\label{eq:local-isometry}
  p_g(y,z)=d(y,z)\qquad(y,z\in U_x).
\end{equation}
Consequently, $\tau_{p_g}\in\mathcal L^m_a(\tau_d)$. More precisely,
a set $U\subseteq X$ is $\tau_{p_g}$-open if and only if it is
$\tau_d$-open and either $a\notin U$, or $a\in U$ and
$\{g<\varepsilon\}\subseteq U$ for some $\varepsilon>0$.
If $d$ is complete, then $p_g$ is complete.
\end{lemma}

\begin{proof}
Symmetry and the identity of indiscernibles are immediate. To prove
the triangle inequality, fix $x,y,z\in X$. The quantity $p_g(x,z)$ is
at most each of the following four numbers:
\[
\begin{aligned}
  &d(x,y)+d(y,z),
  &&d(x,y)+g(y)+g(z),\\
  &g(x)+g(y)+d(y,z),
  &&g(x)+2g(y)+g(z).
\end{aligned}
\]
The first bound uses the triangle inequality for $d$. The middle two
use $g(x)\le d(x,y)+g(y)$ and $g(z)\le g(y)+d(y,z)$, respectively.
The last uses nonnegativity. The minimum of these four numbers is
$p_g(x,y)+p_g(y,z)$, proving the required inequality.

Since $g(a)=0$ and $g$ is $1$-Lipschitz, $g(x)\le d(a,x)$; hence
$p_g(a,x)=g(x)$. The inequality $p_g\le d$ makes
$\id:(X,d)\to(X,p_g)$ continuous. If $x\ne a$, then
$g(x)+g(y)\ge g(x)>g(x)/2$ for every $y$, which proves
\eqref{eq:equal-ball}. For $y,z\in U_x$, the Lipschitz inequality gives
$g(y)>g(x)/2$ and $g(z)>g(x)/2$, whereas $d(y,z)<g(x)$. Thus the
minimum in \eqref{eq:minmetric} equals $d(y,z)$, proving
\eqref{eq:local-isometry}. This also proves equality of the two
topologies away from $a$. The description of open sets follows from
$B_{p_g}(a,\varepsilon)=\{g<\varepsilon\}$ and the fact that
$X\setminus\{a\}$ is open in both topologies.

Suppose finally that $d$ is complete, and let $(y_k)$ be $p_g$-Cauchy.
For all $x,y$,
\[
  |g(x)-g(y)|\le\min\{d(x,y),g(x)+g(y)\}=p_g(x,y).
\]
Thus $g(y_k)\to L$ for some $L\ge0$. If $L=0$, then
$p_g(a,y_k)=g(y_k)\to0$. If $L>0$, then eventually $g(y_k)\ge L/2$.
Given $\varepsilon>0$, for sufficiently large $k,m$ we also have
$p_g(y_k,y_m)<\min\{L,\varepsilon\}$. Since
$g(y_k)+g(y_m)\ge L$, it follows that
$d(y_k,y_m)=p_g(y_k,y_m)<\varepsilon$. Thus $(y_k)$ is $d$-Cauchy
and converges to some $y\in X$. Since $p_g\le d$, it converges to $y$
in $p_g$ as well.
\end{proof}

\begin{proof}[Proof of Theorem~\ref{thm:main-existence}]
Such sets $D$ exist by Lemma~\ref{lem:discrete}. Fix any countably
infinite closed discrete set $D=\{x_n:n\in\N\}\subseteq X\setminus\{a\}$,
and let $(\lambda_n)$ be the prescribed positive null sequence. Define $h,g,p$
by \eqref{eq:main-formula}. Each function
$x\mapsto\lambda_n+d(x,x_n)$ is $1$-Lipschitz. Their infimum is finite
and nonnegative, and satisfies
\[
  h(x)\le h(y)+d(x,y),
\]
so $h$ is $1$-Lipschitz.

If $x\notin D$, then $h(x)\ge d(x,D)>0$ because $D$ is closed. If
$x=x_m$, choose $\eta_m>0$ with $d(x_m,x_n)\ge\eta_m$ for $n\ne m$.
Then
\[
  h(x_m)\ge\min\{\lambda_m,\eta_m\}>0.
\]
Thus $h$ is positive everywhere, while $h(x_n)\le\lambda_n$.

The minimum of two real-valued $1$-Lipschitz functions is
$1$-Lipschitz. Therefore $g=\min\{d(\cdot,a),h\}$ is
$1$-Lipschitz, $g^{-1}(0)=\{a\}$, and $g(x_n)\le\lambda_n$.
Lemma~\ref{lem:minmetric} proves all the metric, localization, and
completeness assertions. Finally, $p(a,x_n)\to0$, while the
$\tau$-open neighborhood $X\setminus D$ of $a$ contains no $x_n$.
Thus $(x_n)$ does not converge to $a$ in $\tau$, and
$\tau_p\subsetneq\tau$.
\end{proof}

\begin{proposition}[The greatest metric under the distance bounds]
\label{prop:maximal-metric}
With the data of Theorem~\ref{thm:main-existence}, the metric $p$ in
\eqref{eq:main-formula} is the greatest, in pointwise order, among
all metrics $r$ on $X$ satisfying
\[
  r\le d,\qquad r(a,x_n)\le\lambda_n\quad(n\in\N).
\]
\end{proposition}

\begin{proof}
Let $r$ satisfy these conditions. For every $x\in X$ and $n\in\N$,
\[
  r(a,x)\le d(a,x),\qquad
  r(a,x)\le r(a,x_n)+r(x_n,x)\le\lambda_n+d(x_n,x).
\]
Taking the infimum over $n$ gives $r(a,x)\le g(x)$. Therefore
\[
  r(x,y)\le\min\{d(x,y),r(x,a)+r(a,y)\}
           \le\min\{d(x,y),g(x)+g(y)\}=p(x,y).
\]
Theorem~\ref{thm:main-existence} shows that $p$ itself satisfies all
the stated bounds.
\end{proof}

\begin{remark}[Realization as a classical metric quotient]\label{rem:quotient}
The connection with Hausdorff's construction can be made exact.
Equip $X\times[0,\infty)$ with
\[
  m\bigl((x,r),(y,t)\bigr)=d(x,y)+|r-t|,
\]
and set
\[
  F=\{(a,0)\}\cup\{(x_n,\lambda_n):n\in\N\},
  \qquad Z=(X\times\{0\})\cup F.
\]
The set $D\cup\{a\}$ is closed discrete in $X$. The set $F$ is the
graph of a continuous function on this closed subspace, so $F$ is
closed in $X\times[0,\infty)$; hence $Z$ is closed as well. Also
$F\cap(X\times\{0\})=\{(a,0)\}$.

On the set $Z/F$ obtained by identifying $F$ to a single point, use
the classical metric quotient formula
\cite[p.~354]{Hausdorff}, \cite[Proposition~2.6]{Ishiki}:
\[
  Q([u],[v])=\min\{m(u,v),m(u,F)+m(v,F)\}\qquad(u,v\in Z).
\]
If one representative lies in $F$, the formula reduces to the distance
of the other representative from $F$, so it is well defined.
The map $x\mapsto[(x,0)]$ is a bijection from $X$ onto $Z/F$, and
\[
\begin{split}
  m((x,0),F)
  &=\min\left\{d(x,a),
           \inf_{n\in\N}\bigl(d(x,x_n)+\lambda_n\bigr)\right\}\\
  &=g(x).
\end{split}
\]
Thus pulling back $Q$ gives exactly $p$ in \eqref{eq:main-formula}.
Here $Z/F$ carries the displayed metric; agreement with the usual
topological quotient topology is not assumed.

If $d$ is complete, then $Z$ is complete as a closed subspace of the
complete product. The completeness of $p$ therefore also follows from
\cite[Proposition~2.9]{Ishiki}. For a general gauge $g$ as in
Lemma~\ref{lem:minmetric}, the analogous realization uses
$F_g=\{(x,g(x)):x\in X\}$ and
$Z_g=(X\times\{0\})\cup F_g$: the $1$-Lipschitz property gives
\[
  m((x,0),F_g)=\inf_{y\in X}\{d(x,y)+g(y)\}=g(x).
\]
These realizations identify both the lemma and its application with the
classical construction. Proposition~\ref{prop:maximal-metric} records
the elementary extremal property of the specified choice of $g$.
\end{remark}

\begin{remark}[Meaning of the rate bound]\label{rem:rate}
The numbers $\lambda_n$ prescribe an upper bound relative to the
constructed metric $p$, which itself depends on $(\lambda_n)$. No
intrinsic topological rate or lower bound is asserted. The construction
achieves this bound together with $p\le d$ and exact local preservation
of the given metric away from $a$.
\end{remark}

\section{Continuous gauges and the order of coarsenings}\label{sec:gauges}

Throughout this section, $d$ is a compatible metric for $\tau$ satisfying
$d\le1$. Such a metric is always available: replace any compatible
metric $d_0$ by $d_0/(1+d_0)$. This replacement preserves completeness
when $d_0$ is complete, since the two metrics have the same Cauchy
sequences. The normalization is part of the hypotheses for the cone
formula; mere boundedness, without a diameter restriction, is not enough.

\subsection{Metric realization}

Let $C(X)$ denote the quotient \emph{set}
\[
  C(X)=(X\times[0,\infty))/(X\times\{0\}),
\]
and write $[x,r]$ for the class of $(x,r)$. All classes with $r=0$
are the same vertex.

\begin{lemma}[Standard cone metric]\label{lem:cone}
The formula
\begin{equation}\label{eq:cone}
  \delta([x,r],[y,t])=|r-t|+\min\{r,t\}d(x,y)
\end{equation}
defines a metric on $C(X)$.
\end{lemma}

\begin{proof}
The formula is well defined at the vertex, and symmetry and separation
are immediate. For the triangle inequality, consider $[x,r]$, $[y,t]$,
and $[z,q]$, and put $m=\min\{r,q\}$. If $t\ge m$, then
$\min\{r,t\}\ge m$ and $\min\{t,q\}\ge m$; the triangle inequalities
for $d$ and the absolute value give the result.

If $t<m$, then
\[
\begin{split}
  m d(x,z)
  &\le m\bigl(d(x,y)+d(y,z)\bigr)\\
  &\le t\bigl(d(x,y)+d(y,z)\bigr)+2(m-t),
\end{split}
\]
because $d\le1$. Since
\[
  |r-t|+|t-q|-|r-q|=2(m-t),
\]
the triangle inequality follows in this case as well.
\end{proof}

Formula \eqref{eq:cone} is the standard construction noted in
\cite[Section~4]{BZ}; the proof is included to specify the normalization
and to keep the metric arguments self-contained.

\begin{proposition}[Gauge localization]\label{prop:gauge-localization}
For $s\in\Ga$, formula \eqref{eq:cone-pullback} defines a metric
$\rho_s$. Its topology $\tau_s$ satisfies
\[
  \tau_s\subseteq\tau,
  \qquad \tau_s|_{X\setminus\{a\}}=\tau|_{X\setminus\{a\}},
\]
and
\begin{equation}\label{eq:sublevels}
  B_{\rho_s}(a,\varepsilon)=\{s<\varepsilon\}.
\end{equation}
A set $U\subseteq X$ belongs to $\tau_s$ if and only if $U\in\tau$
and either $a\notin U$, or $a\in U$ and
$\{s<\varepsilon\}\subseteq U$ for some $\varepsilon>0$.
In particular, $\tau_s$ is independent of the compatible metric $d\le1$
used in its construction.
\end{proposition}

\begin{proof}
The map $x\mapsto[x,s(x)]$ is injective because $s^{-1}(0)=\{a\}$.
Pulling back \eqref{eq:cone} gives $\rho_s$, so it is a metric.
Continuity of $s$ and the estimate
\[
  \rho_s(x,y)\le |s(x)-s(y)|+s(x)d(x,y)
\]
show that $\id:(X,\tau)\to(X,\tau_s)$ is continuous.
For $x\ne a$, if $\rho_s(x,y)<s(x)/2$, then $s(y)>s(x)/2$ and
\begin{equation}\label{eq:cone-local-lower}
  \rho_s(x,y)\ge\frac{s(x)}2 d(x,y).
\end{equation}
Thus the inverse identity is continuous at every $x\ne a$.

Since $s(a)=0$, we have $\rho_s(a,x)=s(x)$, proving
\eqref{eq:sublevels}. If $U\in\tau_s$, then $U\in\tau$, and a set
containing $a$ must contain a sublevel set. Conversely, a $\tau$-open
set avoiding $a$ is open in the common open subspace
$X\setminus\{a\}$, hence is $\tau_s$-open in $X$. If $U\in\tau$
contains $a$ and contains a sublevel set, then $a$ is a $\tau_s$-interior
point, and every other point of $U$ is a $\tau_s$-interior point by
the already proved local agreement. This proves the description of open
sets, and hence independence of $d$.
\end{proof}

\begin{proposition}[Completeness of the specified cone metric]\label{prop:cone-complete}
If $d\le1$ is complete and $s\in\Ga$, then $\rho_s$ is complete.
\end{proposition}

\begin{proof}
Let $(y_k)$ be $\rho_s$-Cauchy. Since
$|s(y_k)-s(y_m)|\le\rho_s(y_k,y_m)$, there is $L\ge0$ with
$s(y_k)\to L$. If $L=0$, then $\rho_s(a,y_k)=s(y_k)\to0$.
If $L>0$, then eventually
\[
  d(y_k,y_m)\le \frac{2}{L}\rho_s(y_k,y_m).
\]
Thus $(y_k)$ is $d$-Cauchy and converges to some $y\in X$. Continuity
of $s$ gives $s(y_k)\to s(y)$, and the defining formula yields
$\rho_s(y_k,y)\to0$.
\end{proof}

\begin{proposition}[Representation]\label{prop:representation}
Let $\sigma\in\La$, let $q$ be any compatible metric for $\sigma$,
and set $s(x)=q(x,a)$. Then $s\in\Ga$ and $\tau_s=\sigma$.
\end{proposition}

\begin{proof}
The function $s$ is $\sigma$-continuous, hence $\tau$-continuous because
$\sigma\subseteq\tau$, and it vanishes exactly at $a$. The two
topologies $\tau_s$ and $\sigma$ agree away from $a$. At $a$,
\[
  B_{\rho_s}(a,\varepsilon)=\{s<\varepsilon\}=B_q(a,\varepsilon).
\]
Their neighborhood bases therefore agree at every point, so the
topologies are equal.
\end{proof}

\subsection{Cofinal order and finite lattice operations}

For $s,t\in\Ga$, define
\begin{equation}\label{eq:preorder}
  s\preceq t
  \quad\Longleftrightarrow\quad
  \forall\varepsilon>0\ \exists\delta>0:
       \{t<\delta\}\subseteq\{s<\varepsilon\}.
\end{equation}
Write $s\asymp t$ if both $s\preceq t$ and $t\preceq s$. Thus
$s\asymp t$ means that the two sublevel families are mutually cofinal.

\begin{theorem}[Order representation]\label{thm:order}
For $s,t\in\Ga$,
\[
  \tau_s\subseteq\tau_t\quad\Longleftrightarrow\quad s\preceq t.
\]
Consequently, $\tau_s=\tau_t$ if and only if $s\asymp t$, and
\[
  \Ga/{\asymp}\longrightarrow\La,
  \qquad [s]\longmapsto\tau_s,
\]
is an order isomorphism, with the quotient order induced by $\preceq$.
\end{theorem}

\begin{proof}
If $\tau_s\subseteq\tau_t$, then each $\{s<\varepsilon\}$ is a
$\tau_t$-neighborhood of $a$ and contains some $\{t<\delta\}$.
Conversely, suppose $s\preceq t$ and $U\in\tau_s$. If $a\notin U$,
then $U$ is open in the common topology away from $a$, hence is
$\tau_t$-open. If $a\in U$, choose $\varepsilon>0$ with
$\{s<\varepsilon\}\subseteq U$, and then $\delta>0$ with
$\{t<\delta\}\subseteq\{s<\varepsilon\}$. The open-set description
in Proposition~\ref{prop:gauge-localization} gives $U\in\tau_t$.
Equality is mutual inclusion, and surjectivity follows from
Proposition~\ref{prop:representation}.
\end{proof}

\begin{corollary}[A distributive lattice]\label{cor:lattice}
The family $\La$ is closed under nonempty finite joins and meets in
the lattice of all topologies on $X$. For $s,t\in\Ga$,
\begin{equation}\label{eq:lattice}
  \tau_s\vee\tau_t=\tau_{\max(s,t)},
  \qquad
  \tau_s\wedge\tau_t=\tau_{\min(s,t)}.
\end{equation}
It is a distributive lattice with greatest element $\tau$.
\end{corollary}

\begin{proof}
Put $u=\max(s,t)$ and $v=\min(s,t)$; both lie in $\Ga$.
The identities
\[
  \{u<\varepsilon\}=\{s<\varepsilon\}\cap\{t<\varepsilon\},
  \qquad
  \{v<\varepsilon\}=\{s<\varepsilon\}\cup\{t<\varepsilon\}
\]
give the required local bases. More explicitly, $\tau_u$ contains
both $\tau_s$ and $\tau_t$. Conversely, if $U\in\tau_u$ contains $a$,
it contains an intersection
$\{s<\varepsilon\}\cap\{t<\varepsilon\}$, which is open in
$\tau_s\vee\tau_t$; also $U\setminus\{a\}$ is open in both
topologies. Thus $U$ is open in the join. Sets avoiding $a$ cause no
additional condition.

The meet of two topologies is their intersection. If $a\in U$ and
$U\in\tau_s\cap\tau_t$, then $U$ contains both
$\{s<\varepsilon\}$ and $\{t<\delta\}$ for some
$\varepsilon,\delta>0$. With $\eta=\min\{\varepsilon,\delta\}$,
it contains $\{v<\eta\}$. Conversely, a set containing a sublevel
set of $v$ contains a sublevel set of each of $s$ and $t$. The
open-set description proves the meet identity.

Distributivity follows from the pointwise distributive identities for
minimum and maximum. Finally, $r(x)=d(x,a)$ belongs to $\Ga$ and
$\tau_r=\tau$, since its sublevel sets are the original balls at $a$.
\end{proof}

\begin{example}[Two incomparable coarsenings]\label{ex:lattice}
Let $X=\{a\}\cup\{u_n:n\in\N\}\cup\{v_n:n\in\N\}$ consist of
distinct points and have the discrete topology. Define
\[
\begin{array}{c|ccc}
   &a&u_n&v_n\\ \hline
 s &0&1/n&1\\
 t &0&1&1/n
\end{array}
\]
Then $u_n\to a$ in $\tau_s$ but not in $\tau_t$, and $v_n\to a$
in $\tau_t$ but not in $\tau_s$. Thus the two topologies are
incomparable. Their join is discrete because $\max(s,t)=1$ away
from $a$. In their meet, a neighborhood of $a$ contains all but
finitely many points of $X\setminus\{a\}$; hence the meet is the
one-point compactification of that countable discrete space.
\end{example}

\begin{proof}[Proof of Theorem~\ref{thm:main-gauges}]
The metric and local assertions are
Proposition~\ref{prop:gauge-localization}, representation is
Proposition~\ref{prop:representation}, the order criterion is
Theorem~\ref{thm:order}, and completeness is
Proposition~\ref{prop:cone-complete}.
\end{proof}

\section{Extension traces, strictness, and preservation}\label{sec:structure}

\subsection{The relation to one-point extensions}

Put $Y=X\setminus\{a\}$ with its $\tau$-subspace topology.
Recall that an \emph{extension trace} on $Y$ is a sequence of nonempty
open sets $(U_n)$ such that
\begin{equation}\label{eq:trace}
  \clos{U_{n+1}}^{\,Y}\subseteq U_n,
  \qquad \bigcap_{n\ge1}U_n=\varnothing.
\end{equation}
Repeated sets can be discarded. Such a trace defines a one-point
metrizable extension by using $\{a\}\cup U_n$ as a neighborhood base
at the added point. This is Alexander's metrization theorem as stated
in \cite[Theorem~3.3]{HJW}. The cofinal comparison and the corresponding
order representation are \cite[Theorems~3.5 and~3.9]{HJW}. These trace results apply to general
metrizable spaces and do not require local compactness.

\begin{proposition}[The coarsening constraint on traces]\label{prop:traces}
Suppose $s\in\Ga$ and $a$ is nonisolated in $\tau_s$. Then
\[
  U_n(s)=\{x\in Y:s(x)<2^{-n}\}
\]
is an extension trace, after discarding repetitions if necessary, and
generates $(X,\tau_s)$. Conversely, a trace $(U_n)$ on $Y$ generates
a topology in $\La$ if and only if $\{a\}\cup U_n\in\tau$ for
every $n$.
\end{proposition}

\begin{proof}
Nonisolation makes every $U_n(s)$ nonempty. Continuity and positivity
of $s$ on $Y$ give
\[
  \clos{U_{n+1}(s)}^{\,Y}
  \subseteq\{x\in Y:s(x)\le2^{-(n+1)}\}
  \subseteq U_n(s),
  \qquad\bigcap_nU_n(s)=\varnothing.
\]
The trace generates $\tau_s$ by \eqref{eq:sublevels}.

For the converse, let $\sigma$ be the metrizable extension topology
defined by $(U_n)$. If $\sigma\subseteq\tau$, its basic sets
$\{a\}\cup U_n$ are $\tau$-open. If all these basic sets are
$\tau$-open, then every $\sigma$-open set avoiding $a$ is open in
$Y$, hence in $\tau$. A $\sigma$-open set $W$ containing $a$ is the
union of some $\{a\}\cup U_n\subseteq W$ and the $\tau$-open set
$W\cap Y$. Thus $\sigma\subseteq\tau$.
\end{proof}

\begin{remark}\label{rem:extension-comparison}
If $a$ is nonisolated in $\tau$, it is nonisolated in every
$\sigma\in\La$, so all these topologies are extensions of the same
dense subspace $Y$. If $a$ is isolated in $\tau$, every one-point
metrizable extension of $Y$ is coarser than the topological sum
$Y\sqcup\{a\}$, which is $\tau$; the only member of $\La$ in which
$a$ remains isolated is $\tau$ itself. In both cases, $\tau$ is an
upper bound on the allowed extension topologies. Theorem~\ref{thm:order}
expresses their cofinal comparison through continuous scalar gauges,
and Proposition~\ref{prop:traces} specifies the constraint imposed
by $\tau$.
\end{remark}

\subsection{Closed-discrete witnesses and compactness}

\begin{theorem}[Strictness criterion]\label{thm:strictness}
Let $\sigma\in\La$, and let $s\in\Ga$ represent $\sigma$.
The following are equivalent:
\begin{enumerate}
\item $\sigma\subsetneq\tau$;
\item some $\tau$-open neighborhood $U$ of $a$ satisfies
      $\inf_{x\in X\setminus U}s(x)=0$;
\item there is a countably infinite $\tau$-closed discrete set
      $D=\{x_n:n\in\N\}\subseteq X\setminus\{a\}$ with
      $x_n\to a$ in $\sigma$.
\end{enumerate}
Equivalently, condition \textup{(ii)} says that, for any compatible
metric $d_0$ of $\tau$, there is an $r>0$ such that
\[
  \inf\{s(x):d_0(x,a)\ge r\}=0.
\]
Here the infimum of the empty set is understood to be $+\infty$.
\end{theorem}

\begin{proof}
By Proposition~\ref{prop:gauge-localization}, equality $\sigma=\tau$
holds exactly when every $\tau$-open neighborhood of $a$ contains a
sublevel set of $s$. This is equivalent to
$\inf_{X\setminus U}s>0$ for every such $U$, proving
\textup{(i)}$\Longleftrightarrow$\textup{(ii)}.

Assume \textup{(ii)}. Choose $y_n\in X\setminus U$ with
$s(y_n)<1/n$. Then $y_n\to a$ in $\sigma$. Because every point of
$X\setminus U$ has positive $s$-value, the range is infinite and each
point occurs only finitely often. Passing to a subsequence yields
distinct points $x_n\in X\setminus U$ with $x_n\to a$ in $\sigma$.
If their range $D$ had a $\tau$-accumulation point $b$, then
$b\in X\setminus U$ and $b\ne a$. Metrizability would give a
subsequence converging to $b$ in $\tau$, hence in $\sigma$. This
contradicts uniqueness of the limit in the Hausdorff space $(X,\sigma)$.
Thus $D$ is $\tau$-closed discrete.

If \textup{(iii)} holds, the $\tau$-open neighborhood $X\setminus D$
of $a$ prevents convergence to $a$ in $\tau$, so $\sigma\ne\tau$.
Finally, choose $r>0$ with $B_{d_0}(a,r)\subseteq U$ to obtain the
metric version of \textup{(ii)}. Its converse follows by taking
$U=B_{d_0}(a,r)$.
\end{proof}

\begin{corollary}[Localized compactness characterization]\label{cor:compactness}
For a metrizable space $(X,\tau)$ and a prescribed $a\in X$, the
following are equivalent:
\begin{enumerate}
\item $(X,\tau)$ is noncompact;
\item there is a strict member of $\La$.
\end{enumerate}
Consequently, a metrizable topology admits no strictly coarser
metrizable topology if and only if it is compact.
\end{corollary}

\begin{proof}
Theorem~\ref{thm:main-existence} proves \textup{(i)}$\Longrightarrow$
\textup{(ii)}. Conversely, Theorem~\ref{thm:strictness} produces an
infinite closed discrete subset, which is impossible in a compact
space. For the last assertion, a continuous bijection from a compact
space to a Hausdorff space is a homeomorphism; the noncompact direction
is already localized at the prescribed point.
\end{proof}

\subsection{Standard preservation consequences}

The next proposition isolates properties determined by agreement off
one point. Its completeness part is the finite-remainder argument of
\cite[Theorem~2.2]{HJW}; the proof below also covers the case in which
the punctured subspace is not dense.

\begin{proposition}[Agreement off one point]\label{prop:preservation}
Let $\tau$ and $\sigma$ be metrizable topologies on $X$ with
$\tau|_{X\setminus\{a\}}=\sigma|_{X\setminus\{a\}}$.
No inclusion between them is assumed. Then:
\begin{enumerate}
\item $\Bor(X,\tau)=\Bor(X,\sigma)$;
\item $(X,\tau)$ is completely metrizable if and only if
      $(X,\sigma)$ is completely metrizable;
\item $(X,\tau)$ is Polish if and only if $(X,\sigma)$ is Polish.
\end{enumerate}
\end{proposition}

\begin{proof}
Put $Y=X\setminus\{a\}$. If $U\in\tau$, then $U\cap Y$ is
$\sigma$-open, because $Y$ is $\sigma$-open and has the common subspace
topology. The set $U$ differs from $U\cap Y$ by at most the closed
singleton $\{a\}$, so $U$ is $\sigma$-Borel. Interchanging $\tau$ and
$\sigma$ proves \textup{(i)}.

Suppose $(X,\tau)$ is completely metrizable. Its open subspace $Y$
is completely metrizable. Embed $(X,\sigma)$ in the completion $Z$
of any compatible metric. By the standard $G_\delta$ characterization
of complete metrizability, $Y$ is $G_\delta$ in $Z$. The singleton
$\{a\}$ is also $G_\delta$ in the metric space $Z$. Finite unions of
$G_\delta$ sets are $G_\delta$, so $X=Y\cup\{a\}$ is $G_\delta$ in
$Z$ and is completely metrizable. Symmetry proves \textup{(ii)}.

If $(X,\tau)$ is Polish, its open subspace $Y$ is separable. A
countable dense subset of $Y$, together with $a$, is dense in
$(X,\sigma)$. Combining separability with \textup{(ii)} proves
\textup{(iii)}, again by symmetry.
\end{proof}

In particular, every member of $\La$ has the same Borel structure as
$\tau$. If $\tau$ is completely metrizable or Polish, so is every
member of $\La$. For a completely metrizable $\tau$, the explicit
complete metric is also available from
Propositions~\ref{prop:representation} and \ref{prop:cone-complete}:
choose a complete compatible $d\le1$ and use $s(x)=q(x,a)$ for any
compatible metric $q$ of $\sigma$. Together with
Corollary~\ref{cor:compactness}, this yields the corresponding compactness
characterizations using strictly coarser completely metrizable or Polish
topologies.

\section{Consequences and examples}\label{sec:examples}

\subsection{Bounded continuous functions}

For a topology $\nu$ on $X$, let $\Cb(X,\nu)$ denote the real Banach
algebra of bounded continuous functions with the supremum norm, and
write $\ev_x(f)=f(x)$.

\begin{proposition}\label{prop:functions}
Let $\sigma\in\La$ be strict and let $s\in\Ga$ represent
$\sigma$. Then $\Cb(X,\sigma)$ is a proper closed unital subalgebra
of $\Cb(X,\tau)$, consisting exactly of those $f\in\Cb(X,\tau)$
such that
\begin{equation}\label{eq:functions}
  \forall\varepsilon>0\ \exists\eta>0\ \forall x\in X:
  \quad s(x)<\eta\ \Longrightarrow\ |f(x)-f(a)|<\varepsilon.
\end{equation}
If $D=\{x_n:n\in\N\}$ is a closed-discrete witness from
Theorem~\ref{thm:strictness}, then
\[
  \ev_{x_n}\xrightarrow{w^*}\ev_a
       \quad\text{in }\Cb(X,\sigma)^*,
  \qquad
  \|\ev_{x_n}-\ev_a\|=2\quad(n\in\N),
\]
whereas $\ev_{x_n}$ does not converge weak-$*$ to $\ev_a$ on
$\Cb(X,\tau)$.
\end{proposition}

\begin{proof}
The inclusion follows from $\sigma\subseteq\tau$. For a
$\tau$-continuous function, $\sigma$-continuity only needs to be
checked at $a$, where the sublevel sets of $s$ form a local base. This
gives \eqref{eq:functions}. Constants and uniform limits show that the
subalgebra is unital and closed.

Let $d$ be a compatible metric for $\tau$. The function
\[
  r(x)=\frac{d(x,a)}{d(x,a)+d(x,D)}
\]
is bounded and $\tau$-continuous: its denominator is everywhere
positive since $D$ is closed and $a\notin D$. It satisfies $r(a)=0$
and $r(x_n)=1$, so it is not $\sigma$-continuous. This proves
properness and the failure of weak-$*$ convergence on $\Cb(X,\tau)$.

For $f\in\Cb(X,\sigma)$, convergence $x_n\to a$ gives
$f(x_n)\to f(a)$, proving weak-$*$ convergence on the smaller
algebra. Each evaluation has norm one. If $q$ is a compatible metric
for $\sigma$, the function
\[
  u_n(x)=\frac{q(x,a)-q(x,x_n)}{q(x,a)+q(x,x_n)}
\]
is bounded and $\sigma$-continuous, has norm at most one, and takes
the values $-1$ and $1$ at $a$ and $x_n$, respectively. Hence the
norm of the difference of evaluations is exactly $2$.
\end{proof}

These assertions are standard consequences of point convergence and
continuous separation in metric spaces. Their role here is to describe
the effect of the localized change on bounded continuous functions.

\subsection{A non-Archimedean local field}

\begin{example}\label{ex:localfield}
Let $K$ be a non-Archimedean local field, let
$v:K^\times\to\mathbb Z$ be its normalized discrete valuation, let
$\pi$ be a uniformizer, and let $q$ be the cardinality of its residue
field. Normalize $|\pi|=q^{-1}$, so $|x|=q^{-v(x)}$ for $x\ne0$;
see \cite{Serre}. Use the original complete metric
\[
  d(x,y)=|x-y|,
\]
which need not be bounded. Put $a=0$, $x_n=\pi^{-n}$, and
$\lambda_n=q^{-n}$.

If $m>n$, then
\[
  \pi^{-m}-\pi^{-n}=\pi^{-m}(1-\pi^{m-n}),
  \qquad |\pi^{-m}-\pi^{-n}|=q^m,
\]
because $1-\pi^{m-n}$ is a unit. Thus $D=\{\pi^{-n}:n\in\N\}$
is separated and hence closed discrete. For the functions in
\eqref{eq:main-formula}, the term indexed by $n$ gives
$h(x_n)=q^{-n}$: every other term is greater than $1$. Consequently,
\[
  g(x_n)=\min\{q^n,q^{-n}\}=q^{-n},
  \qquad p(0,\pi^{-n})=q^{-n}.
\]
Theorem~\ref{thm:main-existence} therefore yields a complete metric
$p\le d$ inducing a strict localized coarsening $\sigma$ of the
usual topology $\tau_K$, with
\[
  \pi^{-n}\longrightarrow0\quad\text{in }(K,p).
\]
For every $x\ne0$, the metric $p$ agrees with $|y-z|$ for all $y,z$
in a common neighborhood of $x$. The topology $\sigma$ is Polish,
since $K$ is separable, and has the same Borel sets as $\tau_K$.
Every $f\in\Cb(K,\sigma)$ satisfies $f(\pi^{-n})\to f(0)$.
For $K=\mathbb Q_\ell$, where $\ell$ is prime, one may take
$\pi=\ell$ and $q=\ell$.
\end{example}

\begin{remark}[Group topologies]\label{rem:groups}
The topology in Example~\ref{ex:localfield} is a topology on the
underlying set of $K$. It cannot also be a Hausdorff group topology
for the original addition. More generally, if $G$ is a nontrivial
group and $\tau,\sigma$ are Hausdorff group topologies with
\[
  \tau|_{G\setminus\{e\}}=\sigma|_{G\setminus\{e\}},
\]
then $\tau=\sigma$. Indeed, $G\setminus\{e\}$ is open in both
topologies, so their local topologies agree at every $g\ne e$.
Translation by $g^{-1}$ is a homeomorphism for both topologies and
transports this common local base to $e$. Hence the local topologies
agree everywhere. A strict one-point localized coarsening therefore
cannot preserve a Hausdorff group topology.
\end{remark}

\section*{AI Use Declaration}
The authors declare that GPT-5.6 Sol was used solely for language polishing and editorial assistance. All mathematical ideas, proofs, and results presented in this manuscript are entirely the work of the authors.

\end{document}